\documentclass{article}

\usepackage{amsmath}
\usepackage{amssymb}
\usepackage{hyperref}

\newtheorem{thm}{Theorem}[section]

\newtheorem{prop}[thm]{Proposition}

\newcommand{\qed}{\hfill $\fbox{}$}

\newenvironment{proof}{\vskip 0.2in\hspace{-.22 in}\textbf{Proof:}}{\qed}

\title{Quotients of Primes in an Algebraic Number Ring}
\author{Brian D. Sittinger}
\begin{document}
\maketitle

\begin{abstract}
It has been established on many occasions that the set of quotients of prime numbers is dense in the set of positive real numbers. More recently, it has been proved that the set of quotients of primes in the Gaussian integers is dense in the complex plane. In this article, we not only extend this result to any imaginary quadratic number ring, but also prove that the set of quotients of primes in any real quadratic number ring is dense in the set of real numbers. To conclude, we show how to extend these results to an arbitrary algebraic number ring.
\end{abstract}

\normalsize

\section{Introduction}  \label{sect1}

It is a standard fact from Real Analysis that $\mathbb{Q}$ is a dense subset of $\mathbb{R}$. However, it is not as well-known that the set of quotients of \emph{prime} numbers is a dense subset of $\mathbb{R}_{> 0}$. One of the earliest appearances of this fact is in Sierpi\'{n}ski's textbook on Number Theory \cite{Sierpinski}. This fact, along with a few variations of it, has been studied and explored in great detail through the years; see \cite{Brown}, \cite{Garcia1}, \cite{Hobby}, \cite{Micholson}, and \cite{Starni} for instance. These results are proved by using some variant of the Prime Number Theorem, which as a reminder we now state in its classic formulation below.

\begin{thm}
If $\pi(x)$ denotes the number of prime numbers in $\mathbb{N}$ less than or equal to some $x > 0$, then
$\displaystyle \pi(x) \sim \frac{x}{\ln{x}}$. In other words, $\displaystyle\lim_{x \to \infty} \frac{\pi(x)}{\frac{x}{\ln{x}}} = 1.$
\end{thm}

Recently, Garcia (\cite{Garcia}) established that the set of quotients of primes from the ring of \emph{Gaussian integers} $\mathbb{Z}[i]$ is dense in $\mathbb{C}$ by using a suitable generalization of the Prime Number Theorem to $\mathbb{Z}[i]$. We generalize his work by showing that the set of quotients of primes from any fixed imaginary quadratic ring is dense in $\mathbb{C}$. This method of proof also permits use to show with little effort that the set of quotients of primes belonging to a fixed congruence class of Gaussian integers is dense in $\mathbb{C}$.

After accomplishing this task, we turn our attention to real quadratic number rings, where we show that the set of quotients of primes from a given real quadratic ring is dense in $\mathbb{R}$. Finally, we show how to establish analogous density results for the set of quotients of primes from any algebraic number ring.

In this article, we refer to the prime numbers in $\mathbb{N}$ as \emph{rational primes}, in order to distinguish them from the prime elements in a given algebraic number ring.

\section{Background on Algebraic Number Rings}

We first fix some notation for an algebraic number ring. For further details, check any textbook on Algebraic Number Theory, such as \cite{Marcus} and \cite{Stewart}.

Let $K$ denote an algebraic number field over $\mathbb{Q}$ (that is, $[K : \mathbb{Q}] < \infty$) with $\mathcal{O}$ being its corresponding number ring. In particular, a quadratic number field has the form $\mathbb{Q}(\sqrt{d})$ for any square-free integer $d$. Its ring of integers $\mathcal{O}$, called a \emph{quadratic number ring}, is the set $\{a + b\omega : a, b \in \mathbb{Z}\},$ where
$$\omega = \begin{cases}
\sqrt{d} & \text{if} \; d \not\equiv 1 \; \text{mod 4}\\
\frac{1+\sqrt{d}}{2} & \text{if} \; d \equiv 1 \; \text{mod 4}.\\
\end{cases}$$

Although we do not necessarily have unique factorization into irreducible elements in $\mathcal{O}$, we do have unique factorization into prime \emph{ideals} in $\mathcal{O}$. One way to `measure' how far $\mathcal{O}$ is from being a UFD is as follows: we set up an equivalence relation on the ideals in $\mathcal{O}$ as follows: we say that two ideals $\mathfrak{a}$ and $\mathfrak{b}$ in $\mathcal{O}$ are equivalent if there exist nonzero $\alpha, \beta \in \mathcal{O}$ such that $\langle \alpha \rangle \mathfrak{a} = \langle \beta \rangle \mathfrak{b}$. This equivalence partitions the ideals into disjoint ideal classes that form a finite abelian group called the \emph{class group} of $\mathcal{O}$. Its cardinality is called the \emph{class number} of $\mathcal{O}$ and is denoted by $h$. For what we need, it should be noted that the trivial ideal class is the class of \emph{principal} ideals, and thus $\mathcal{O}$ is a UFD iff $h=1$ (as all ideals in a UFD are principal). The trivial ideal class is of chief interest to us, because a generator of a principal prime ideal gives a prime element in $\mathcal{O}$.

\section{Quotients of Primes in an Imaginary Quadratic Ring}

In this section, we extend the density result for the Gaussian integers \cite{Garcia} to any imaginary quadratic number ring $\mathcal{O}$. To do this, we use the following `angular' prime number theorem of Kubilius \cite{Kubilius2} and Hecke \cite{Hecke}. It states that any sector $\{z \in \mathbb{C} \, : \, \theta_1 < \arg{z} < \theta_2\}$ in the complex plane contains prime elements of arbitrarily large magnitude.

\begin{prop}
Let $\mathcal{O}$ be an imaginary quadratic number ring. Fix $0 \leq \theta_1 < \theta_2 \leq 2\pi$ and $x > 0$, and let $\Pi(x; \theta_1, \theta_2)$ denote the number of prime elements $\rho \in \mathcal{O}$ satisfying $|\rho|^2 \leq x$ and $\theta_1 < \arg{\rho} < \theta_2$. Then,

$$\Pi(x; \theta_1, \theta_2) \sim \frac{g (\theta_2 - \theta_1)}{2\pi h} \cdot \frac{x}{\ln{x}},$$
where $g$ denotes the number of units in $\mathcal{O}$.
\end{prop}

The presence of $h$ in this prime number formula signifies that we are looking at the principal ideal class. Moreover, the significance of $g$ follows from the fact that a generator of a principal ideal is unique up to multiplication by a unit in $\mathcal{O}$. It should also be noted that $g$ is always finite in an imaginary quadratic number ring.

A consequence of this proposition we will use is that if $0 < a < b$, then for all sufficiently large real $x$, an annular sector $$\{z \in \mathbb{C} : ax < |z|^2 < bx \;\text{and}\; \theta_1 < \arg{z} < \theta_2\},$$
\noindent contains a prime element in $\mathcal{O}$. This follows from straightforward calculus:
\begin{align*}
\lim_{x \to \infty} [\Pi(bx; \theta_1, \theta_2) - \Pi(ax; \theta_1, \theta_2)] & = \lim_{x \to \infty} \Pi(bx; \theta_1, \theta_2) \Big[1 - \frac{\Pi(ax; \theta_1, \theta_2)}{\Pi(bx; \theta_1, \theta_2)}\Big] \\
& = \lim_{x \to \infty} \Pi(bx; \theta_1, \theta_2) \Big[1 - \frac{ax \ln(bx)}{bx \ln(ax)}\Big] \\
& =  \Big(1 - \frac{a}{b}\Big) \lim_{x \to \infty} \Pi(bx; \theta_1, \theta_2)\\
&= \infty.
\end{align*}

Now, we state the desired variant of a prime number theorem to deduce the desired density result for any imaginary quadratic number ring.

\begin{thm} The set of quotients of primes in an imaginary quadratic ring $\mathcal{O}$ is dense in the complex plane.
\end{thm}

\begin{proof}
It suffices to show that any annular sector
$$\{z \in \mathbb{C} : \psi_1 < \arg{z} < \psi_2, \, 0 < r < |z| < R\}$$
contains a quotient of prime numbers in $\mathcal{O}$.

Fixing $\theta \in (0, 2\pi]$, since
$$\lim_{x \to \infty} \Big[\Pi\Big(\frac{x}{r^2}; 0, \theta\Big) - \Pi\Big(\frac{x}{R^2}; 0, \theta\Big)\Big] = \infty,$$

\noindent there exists $x_0 > 0$ such that
$$\Pi\Big(\frac{x}{r^2}; 0, \theta \Big) - \Pi\Big(\frac{x}{R^2}; 0, \theta\Big) \geq 2 \; \text{for all} \, x \geq x_0.$$

Moreover, since the angular prime number theorem implies that there are infinitely many prime numbers in $\mathcal{O}$ in the sector $(\psi_1, \psi_2)$, there exists a prime number $\pi_1$ in the sector $(\psi_1, \psi_2)$ with sufficiently large magnitude ($|\pi_1|^2 > x_0$) such that $$\Pi\Big(\frac{|\pi_1|^2}{r^2}; 0, \xi \Big) - \Pi\Big(\frac{|\pi_1|^2}{R^2}; 0, \xi\Big) \geq 2,$$
where $\xi = \min\{\psi_2 - \arg(\pi_1), \,\arg(\pi_1) - \psi_1\}$.

Next, the inequality in the last assertion implies that there exists a prime number $\pi_2$ satisfying $\displaystyle\frac{|\pi_1|}{R} < |\pi_2| < \frac{|\pi_1|}{r}$ and $0 < \arg(\pi_2) < \xi$.

From this, it now follows that $\displaystyle r < \Big|\frac{\pi_1}{\pi_2}\Big| < R$ and $\displaystyle \psi_1 < \arg\Big(\frac{\pi_1}{\pi_2}\Big) < \psi_2$.
\end{proof}

\section{Quotients of Gaussian Primes, Revisited}

Recall that the ring of \emph{Gaussian integers} $\mathbb{Z}[i]$ is the set of integers in the field $\mathbb{Q}(i)$ and have the form $a + bi$ where $a,b \in \mathbb{Z}$. This set is well-known to be a UFD with units $\pm 1, \pm i$. It can be shown (see \cite{Marcus} or \cite{Stewart}) that a nonzero Gaussian integer is a prime if and only if it is an associate of $p$ (namely $\pm p$ and $\pm pi$) for some rational prime $p \equiv 3$ mod 4, or it is of the form $a+bi \in \mathbb{Z}[i]$ where $a^2+b^2$ is a rational prime. We refer to these primes as \emph{Gaussian primes}.

Note that the proof of Theorem 3.2 applied to $\mathbb{Z}[i]$ does not rely on a specific type of a Gaussian prime, while the proof in \cite{Garcia} fixes the Gaussian prime in the denominator be a rational prime congruent to 3 mod 4. Although this latter idea simplifies the density proof a bit with regards to the argument of the prime, it does not allow for the denominator to be an associate of a Gaussian prime that is not a rational prime.

We exploit this distinction to show how flexible the techniques from our previous density proof is, as we can modify it to give a Gaussian integer analogue of the following fact \cite{Sittinger}:

\begin{prop}
Fix $a, b, m, n \in \mathbb{N}$ such that $\gcd(a, m) = \gcd(b, n) = 1$. Then,
$$\Big\{\frac{p}{q} : p,q \; \text{prime in}\; \mathbb{Z}, p \equiv a \;\text{mod}\; m, \text{ and } q \equiv b \;\text{mod}\; n \Big\}$$ is dense in $\mathbb{R}$.
\end{prop}

To state and prove this generalization, we note that the Gaussian integers have a notion of congruence as well.
Fix $\gamma \in \mathbb{Z}[i]$; for $\alpha, \beta \in \mathbb{Z}[i]$ we write $\alpha \equiv \beta$ mod $\gamma$ iff $\gamma \mid (\beta - \alpha)$. With this notation, we can now state the following `Dirichlet' variant of the prime number theorem \cite{Kubilius}.

\begin{prop}
Fix $0 \leq \theta_1 < \theta_2 \leq 2\pi$ and $x > 0$, and let $\Pi(x; \theta_1, \theta_2; \beta, \gamma)$ denote the number of Gaussian primes $\rho \equiv \beta$ mod $\gamma$ satisfying $\theta_1 < \arg{\rho} < \theta_2$ and $|\rho|^2 \leq  x$. Then,
$$\Pi(x; \theta_1, \theta_2; \beta, \gamma) \sim \frac{4(\theta_2 - \theta_1)}{2\pi \phi(\gamma)} \cdot \frac{x}{\ln{x}},$$
where $\phi(\gamma)$ denotes the number of invertible congruence classes modulo $\gamma$.
\end{prop}

It should be noted that this proposition asserts that the Gaussian primes are equidistributed both by argument and by congruence class. Moreover, the $\phi$ function above is an extension of the Euler phi function to $\mathbb{Z}[i]$.

By adapting the proof of Theorem 3.2 with this variation of the prime number theorem, we immediately obtain the following result.

\begin{thm}
Fix $\beta_1, \beta_2, \gamma_1, \gamma_2 \in \mathbb{Z}[i]$ such that $\gcd(\beta_1, \gamma_1)= \gcd(\beta_2, \gamma_2) = 1$. Then, the set of quotients
$$\Big\{\frac{\pi_1}{\pi_2} : \pi_1, \pi_2 \text{ are Gaussian primes}, \, \pi_1 \equiv \beta_1 \;\text{mod}\; \gamma_1,\;\text{and}\; \pi_2 \equiv \beta_2 \;\text{mod}\; \gamma_2 \Big\}$$ is dense in $\mathbb{C}$.
\end{thm}

\section{Quotients of Primes in a Real Quadratic Number Ring}

In this section, suppose that $\mathcal{O}$ is a real quadratic number ring. These rings are quite different than their imaginary counterparts. First, since a real quadratic number ring is a subset of $\mathbb{R}$, we will be stating a density statement in $\mathbb{R}$ instead of $\mathbb{C}$. Secondly, a real quadratic number ring has \emph{infinitely} many units. Fortunately, any such unit can be written in the form $\pm \eta^k$ for some integer $k$ and fixed unit $\eta \in \mathcal{O}$ called a \emph{fundamental unit} of $\mathcal{O}$. Without loss of generality, we may assume that $\eta > 1$.

There is a Prime Number Theorem for real quadratic number rings that dates back to the work of Hecke \cite{Hecke}, but we will use the following refinement of Rademacher \cite{Rademacher}. To state the theorem, we use the following definition: we say that $\alpha \in \mathcal{O}$ is \emph{totally positive} if both $\alpha$ and its conjugate $\alpha'$ are positive. (Recall that for any $a,b \in \mathbb{Q}$, the conjugate of $a + b\sqrt{d}$ equals $a - b\sqrt{d}$.)

\begin{prop}
Let $\mathcal{O}$ be a real quadratic number ring. Fix $x > 0$, and let $\Pi(x)$ denote the number of totally positive prime elements $\rho \in \mathcal{O}$ satisfying $0 < \rho \leq x$. Then, we have
$$\Pi(x) \sim \frac{1}{2h \ln{\eta}} \int_2^x \int_2^x \frac{du \, dv}{\ln(uv)}.$$
\end{prop}

It should be noted that the double integral is asymptotically equal to $\frac{x^2}{\ln(x^2)}$ \cite{Hinz}. By letting $x \to \infty$, this theorem shows that there are infinitely many totally positive prime elements in $\mathcal{O}$. Now, we use this result to readily establish the density of quotients of primes in $\mathcal{O}$.

\begin{thm}
The set of quotients of primes in a real quadratic number ring $\mathcal{O}$ is dense in $\mathbb{R}$.
\end{thm}

\begin{proof} We need to show that any given interval $(a, b)$ contains a quotient of primes in $\mathcal{O}$. Assume without loss of generality that $a, b > 0$. Then it follows directly from Proposition 5.1 that
$$\lim_{x\to\infty} (\Pi(bx) - \Pi(ax)) = \infty.$$
This implies that for all sufficiently large $x$, there exists a prime $\pi_1 \in \mathcal{O}$ such that $ax < \pi_1 < bx$. Then, since there exist primes in $\mathcal{O}$ with arbitrarily large magnitude by Proposition 5.1, let $x = \pi_2$ for some prime $\pi_2 \in \mathcal{O}$. This immediately yields $a < \frac{\pi_1}{\pi_2} < b$, as required.
\end{proof}

\vspace{.1 in}

\noindent \textbf{Remark:} An alternate proof suggested in private conversation with Garcia uses previous quotient set density theorems and bypasses the use of Proposition 5.1. As a trade-off, it does not allow for the possibility of using non-rational prime numbers to establish the density result and ultimately does not generalize to an arbitrary algebraic number ring. It relies on the following fact \cite{Marcus}:
\begin{center}
\emph{Given a real quadratic number field $\mathbb{Q}(\sqrt{d})$ with $d > 0$ and square-free, an odd rational prime $p$ remains prime in $\mathcal{O}$ iff $d$ is not a square modulo $p$.}
\end{center}

By fixing $d$, this condition gives a set of congruences that $p$ satisfies in terms of $d$. Pick one such congruence condition; by Dirichlet's Theorem on Arithmetic Progressions, there are infinitely many such primes. Applying Proposition 4.1 to this congruence condition yields Theorem 5.2.




\section{Quotients of Primes in an Algebraic Number Ring}
Now, suppose that $K$ is an algebraic number field of degree $n$ over $\mathbb{Q}$ with ring of integers $\mathcal{O}$. Before proceeding further, we need to fix some more notation. Let $n = s + 2t$ and $r = s + t - 1$, where $s$ denotes the number of real embeddings and $s$ denotes the number of pairs of complex embeddings of $K$. For an element $\omega \in \mathcal{O}$, let $\omega^{(j)}$ denote the image of $\omega$ under the $j$th embedding of $K$ into Minkowski space $L_{s,t} = \mathbb{R}^s \times \mathbb{C}^t$ with $j = 1, 2, ..., s+t$. Without loss of generality, assume that the trivial real and complex embeddings give the first and $(s+1)$-th entries of the ordered tuples in $L_{s,t}$.

\vspace{.1 in}

Now, we are able to state a version of Mitsui's generalized Prime Number Theorem \cite{Mitsui} suitable for our purpose.

\begin{thm} Fix $0 \leq \theta_1 < \theta_2 \leq 2\pi$, and let $\Pi(Y; \theta_1, \theta_2)$ denote the number of prime elements $\omega \in \mathcal{O}$ satisfying $|\omega^{(j)}| \leq Y$ for each $j = 1, ...,s+t$ and $\theta_1 <  \arg(\omega^{(k)}) < \theta_2$ for each $k = s+1, ...,s+t$. Then, for all $Y \geq 3$, we have
$$\Pi(Y; \theta_1, \theta_2) \sim \Big(\frac{\theta_2 - \theta_1}{2\pi}\Big)^t \frac{g}{2^s hR} \idotsint\limits_{[2, Y]^s \times [2, Y^2]^t} \frac{dx_1 ... dx_{s+t}}{\ln(x_1 ... x_{s+t})},$$

\noindent where $g$ is the number of roots of unity in $K$, and $R$ is the regulator of $K$.
\end{thm}

\noindent \textbf{Remark:} In the case of a real quadratic number ring ($s = 2$ and $t = 0$), this result reduces to Proposition 5.1, and in the case of an imaginary quadratic number ring ($s = 0$ and $t = 1$), this result reduces to Proposition 3.1.

\vspace{.1 in}

By \cite{Hinz}, the multiple integral is asymptotically equal to $\frac{Y^n}{\ln(Y^n)}$ . Then, by letting $Y \to \infty$, this theorem shows that there are infinitely many prime elements in $\mathcal{O}$ satisfying the desired angular constraints.

Now, we use this result to readily establish the density of quotients of primes in any algebraic number ring $\mathcal{O}$. With this theorem, and using the techniques of proof from the imaginary and real quadratic number rings (Theorems 3.2 and 5.2), we are now able to prove the main result of this paper.

\begin{thm}
Let $K$ be an algebraic number field over $\mathbb{Q}$ with ring of integers $\mathcal{O}$.
If $K \subseteq \mathbb{R}$, then any quotient set of prime elements in $\mathcal{O}$ is dense in $\mathbb{R}$. Otherwise, if $\mathbb{R} \subsetneq K \subseteq \mathbb{C}$, then any quotient set of prime elements in $\mathcal{O}$ is dense in $\mathbb{C}$.
\end{thm}

\noindent \emph{Proof:} First, assume that $K \subseteq \mathbb{R}$. We need to show that there exist a quotient of prime elements in $\mathcal{O}$ in any interval $(a, b)$. Without loss of generality, assume that $a > 0$. Then it follows directly from the generalized Prime Number Theorem that
$$\lim_{x\to\infty} (\Pi(bx; 0, 2\pi) - \Pi(ax; 0, 2\pi)) = \infty.$$
\noindent By considering the first copy of $\mathbb{R}$ in $L_{s,t}$, the last assertion implies that for all sufficiently large $x$, there exists a prime $\pi_1 \in \mathcal{O}$ such that $ax < \pi_1 < bx$. Next, since there exist primes in $\mathcal{O}$ with arbitrarily large magnitude by the generalized Prime Number Theorem, let $x = \pi_2$ for some prime $\pi_2 \in \mathcal{O}$. This immediately yields $a < \frac{\pi_1}{\pi_2} < b$, as required. \\

Next, assume that $\mathbb{R} \subsetneq K \subseteq \mathbb{C}$. It suffices to show that any annular sector $\{z \in \mathbb{C} : \psi_1 < \arg{z} < \psi_2, \, 0 < r < |z| < R\}$ contains a quotient of prime numbers in $\mathcal{O}$. It follows from the generalized Prime Number Theorem that
$$\lim_{x \to \infty} \Big[\Pi\Big(\frac{x}{r}; \psi_1, \psi_2 \Big) - \Pi\Big(\frac{x}{R}; \psi_1, \psi_2 \Big)\Big] = \infty.$$

\noindent Then, there exists $x_0 > 0$ such that
$$\Pi\Big(\frac{x}{r}; \psi_1, \psi_2 \Big) - \Pi\Big(\frac{x}{R}; \psi_1, \psi_2 \Big) \geq 2 \; \text{for all} \, x \geq x_0.$$

\noindent Moreover, since the generalized Prime Number Theorem implies that there are infinitely many prime numbers $\omega \in \mathcal{O}$ satisfying $\theta_1 <  \arg(\omega^{(k)}) < \theta_2$ for each $k = s+1, ...,s+t$, there exists a prime $\pi_1 \in \mathcal{O}$ satisfying $\arg(\pi_1^{(k)})  \in (0, \xi)$ for each $k = s+1, ...,s+t$ with sufficiently large magnitude ($|\pi_1| > x_0$) such that $$\Pi\Big(\frac{|\pi_1|}{r}; 0, \xi \Big) - \Pi\Big(\frac{|\pi_1|}{R}; 0, \xi \Big) \geq 2,$$
where $\xi = \min\{\psi_2 - \arg(\pi_1), \, \arg(\pi_1) - \psi_1)\}$.

Next, by considering the first copy of $\mathbb{C}$ in $L_{s,t}$, the inequality in the last assertion implies that there exists a prime $\pi_2 \in \mathcal{O}$ satisfying $\displaystyle\frac{|\pi_1|}{R} < |\pi_2| < \frac{|\pi_1|}{r}$ and $0 < \arg(\pi_2) < \xi$.

From this, it now follows that $\displaystyle r < \Big|\frac{\pi_1}{\pi_2}\Big| < R$ and $\displaystyle \psi_1 < \arg\Big(\frac{\pi_1}{\pi_2}\Big) < \psi_2$. $\blacksquare$

\section{Concluding Remarks}

A unifying thread through the proofs of the density theorems presented here is having a suitable version of a prime number theorem around to permit us to establish the desired results. The existence of such theorems is rather remarkable in itself and are key results belonging to both analytic and algebraic number theory.

The curious reader may wonder whether there exist congruence class analogues for quotients of primes in quadratic number rings other than $\mathbb{Z}[i]$. This is indeed the case at least when these are UFDs, and the reader is encouraged to find and state the appropriate versions of the prime number theorem to establish these results; they can be found in \cite{Hecke}, \cite{Kubilius2}, and \cite{Rademacher}.

As for another open question, one can try to formulate and establish analogous density results for the noncommutative ring of Hurwitz quaternions
$$H = \Big\{a + bi + cj + dk \; : \; a,b,c,d \in \mathbb{Z} \; \text{or} \; a,b,c,d \in \mathbb{Z} + \frac{1}{2}\Big\},$$ and try to prove that $\{\pi_1 \pi_2^{-1} \; : \; \pi_1, \pi_2 \in H \; \text{are prime}\}$ is dense in $\mathbb{H}$, the set of quaternions over $\mathbb{R}$.

\end{document}